\documentclass[12pt,a4paper,reqno]{amsart}

\usepackage{amssymb,latexsym}
\usepackage{amsmath}
\usepackage{graphicx}
\usepackage{url}	
\usepackage{hyperref}

\usepackage[top=1.6in, bottom=1.3in, left=1.3in, right=1.3in]{geometry}

\allowdisplaybreaks

\DeclareMathOperator*{\lcm}{lcm}
\DeclareMathOperator*{\rad}{rad}

\theoremstyle{plain}
\numberwithin{equation}{section}
\newtheorem{thm}{Theorem}[section]
\newtheorem{theorem}[thm]{Theorem}
\newtheorem{lemma}[thm]{Lemma}
\newtheorem{proposition}[thm]{Proposition}

\allowdisplaybreaks[4]
\everymath{\displaystyle}

\begin{document}

\title[The Distribution of Self-Fibonacci Divisors]{The Distribution of Self-Fibonacci Divisors}
\author{Florian Luca}
\address{School of Mathematics\\
                University of the Witwatersrand\\
                PO Box 2050, Wits, South Africa}
\email{florian.luca@wits.ac.za}

\author{Emanuele Tron}
\address{Scuola Normale Superiore\\
                Piazza dei Cavalieri 7\\
                56126 Pisa\\
                Italy}
\email{emanuele.tron@sns.it}

\subjclass[2010]{11B39}

\begin{abstract}
Consider the positive integers $n$ such that $n$ divides the $n$-th Fibonacci number, and their counting function $A$. We prove that \[A(x)\leq x^{1-(1/2+o(1))\log\log\log x/\log\log x}.\]
\end{abstract}

\maketitle

\section{Introduction}
The Fibonacci numbers notoriously possess many arithmetical properties in relation to their indices. In this context, Fibonacci numbers divisibile by their index constitute a natural subject of study, yet there are relatively few substantial results concerning them in the literature.

Let $\mathcal A= \left \{ a_n \right \}_{n \in \mathbb N}$ be the increasing sequence of natural numbers such that $a_n$ divides $F_{a_n}$: this is OEIS A023172, and it starts \[1, \ 5, \ 12, \ 24, \ 25, \ 36, \ 48, \ 60, \ 72, \ 96, \ 108, \ 120,\  125, \ 144,\ 168, \ 180, \ \ldots \]
(as they have no common name, we dub them \textit{self-Fibonacci divisors}). Let moreover $A(x):= \# \{ n \leq x : n \in \mathcal A \} $ be its counting function.

This kind of sequences has already been considered by several authors; we limit ourselves to mentioning the current state-of-the-art result, due to Alba Gonz\'{a}lez--Luca--Pomerance--Shparlinski.
\begin{proposition}[\cite{A1}, Theorems 1.2 and 1.3] \label{aux}
\[ \left ( \frac 1 4 +o(1) \right ) \log x \leq \log A(x) \leq \log x - (1+o(1)) \sqrt{\log x \log \log x}  .\]
\end{proposition}

We improve the upper bound above as follows.
\begin{theorem} \label{main}
\begin{equation} \label{maineq}
 \log A(x) \leq \log x- \left ( \frac 1 2 +o(1) \right ) \frac{\log x \log \log \log x }{\log \log x}  . \end{equation}
\end{theorem}
The main element of the proof is a new classification of self-Fibonacci divisors.

We now recall some basic facts about Fibonacci numbers. All statements in the next lemma are well-known and readily provable.

\begin{lemma} \label{lem1}
Define $z(n)$ to be the least positive integer such that $n$ divides $F_{z(n)}$ (the \emph{Fibonacci entry point}, or \emph{order of appearance}, of $n$). Then the following properties hold.
\begin{itemize}
\item $z(n)$ exists for all $n \in \mathbb N$. In fact, $z(n) \leq 2n$.
\item $\gcd(F_a, F_b)=F_{\gcd(a,b)}$ for $a, b \in \mathbb N$.
\item $z(p)$ divides $p-\left ( \frac p 5 \right )$ for $p$ prime, $\left ( \frac p 5 \right )$ being the Legendre symbol.
\item If $a$ divides $b$, then $z(a)$ divides $z(b)$. 
\item $z(\lcm(a,b))=\lcm(z(a),z(b))$ for $a, b \in \mathbb N$. In particular, $\lcm(z(a),z(b))$ divides $z(ab)$.
\item $z(p^n)=p^{\max(n-e(p),0)} z(p)$ for $p$ prime, where $e(p):=v_p(F_{z(p)}) \geq 1$ and $v_p$ is the usual $p$-adic valuation.
\end{itemize}
\end{lemma}
From now on, we shall use the above properties without citing them.

Next comes a useful result concerning the $p$-adic valuation of Fibonacci numbers.

\begin{lemma}[\cite{H1}, Theorem 1] \label{lem2}
\[ v_2(F_n)=\begin{cases} 0, &\text{if } n \equiv 1, 2 \pmod 3; \\
1, & \text{if } n \equiv 3 \pmod 6; \\
3, &\text{if } n \equiv 6 \pmod {12}; \\
v_2(n)+2, &\text{if } n \equiv 0 \pmod {12}.
\end{cases}  \]

\[ v_5(F_n)=v_5(n). \]
For $p \neq 2,5$ prime,
\[ v_p(F_n)=\begin{cases} v_p(n)+e(p), &\text{if } n \equiv 0 \pmod {z(p)}; \\
0, &\text{if } n \not \equiv 0 \pmod {z(p)}.
\end{cases}  \]
\end{lemma}

To end the section, we point out an interesting feature of the upper bound in Theorem \ref{main}: it should be, up to a constant factor, best possible.

A squarefree integer $n$ is a self-Fibonacci divisor if and only if $z(p)$ divides $n$ for every prime $p$ that divides $n$. This is certainly true if $p-\left ( \frac p 5 \right )$ divides $n$ for every prime factor $p$ of $n$. This is indeed strongly reminiscent of Korselt's criterion for Carmichael numbers: one should therefore expect heuristics for self-Fibonacci divisors similar to those for Carmichael numbers to be valid; in particular Pomerance's \cite{P}, which would predict
\[ \log A(x) = \log x- \left ( 1 +o(1) \right ) \frac{\log x \log \log \log x }{\log \log x} .\]

\section{Arithmetical Characterisation}

In this section, we show how $\mathcal A$ can be partitioned into subsequences that admit a simple description.

Note that $n$ divides $F_n$ if and only if $z(n)$ divides $n$ and set
\[ \mathcal A_k:= \left \{ n \in \mathbb N : n/z(n)=k \right \}.\]

Our next task is to prove the following characterisation of the $\mathcal A_k$'s. Let $c(k):=\min \mathcal A_k$ whenever $\mathcal A_k$ is not empty.

\begin{thm} \label{char1}
$\mathcal A_k=\varnothing$ if $k$ is divisible by $8$, $5$ or $p^{e(p)+1}$ for an odd prime $p$. Otherwise, if $k=\prod_i p_i ^{\alpha_i}$, 
\begin{itemize}
\item  $\displaystyle \mathcal A_k =\left \{c(k) \cdot 5^{\beta_1} \cdot \prod_{i \atop p_i \neq 2, 5} p_i^{\beta_i} \right \}$ as $(\beta_1, \ldots, \beta_t)$ ranges over $\mathbb N^t$ with the conditions that either $\beta_i=0$ or $\beta_i \geq e(p_i)-v_{p_i}(k)$ if $\alpha_i=e(p_i)$, and $\beta_i=0$ if $\alpha_i<e(p_i)$, for every $i$, if $k$ is odd or $2$ times an odd number;
\item $\displaystyle \mathcal A_k =\left \{c(k) \cdot 5^{\beta_1} \cdot \prod_{i \atop p_i \neq  5} p_i^{\beta_i} \right \}$ with $(\beta_1, \ldots, \beta_t)$ as before, if $k$ is a multiple of $4$.
\end{itemize}

\end{thm}
\begin{proof}

We shall henceforth implicitly assume that \emph{the primes we deal with are distinct from $2$ and $5$}, and all the proofs when some prime is $2$ or $5$ are easily adapted using the modified statement of Lemma \ref{lem2}. 

Suppose that, for some $n$, $n/z(n)=k$, and $p^d$ is the exact power of $p$ that divides $k$. Upon writing $k=p^d k'$ and $n=p^d n'$, with $k'$ coprime to $p$, this becomes $n'/k'=z(p^d n')$. In particular, \[ d+v_p(n')\leq v_p(F_{z(p^dn')})=v_p(F_{n'/k'})\leq v_p(n')+e(p),\] which is absurd if $d \geq e(p)+1$, so that $\mathcal A_k=\varnothing$ if $p^{e(p)+1}$ divides $k$.

Suppose on the other hand that $k$ fulfills the conditions for $\mathcal A_k$ to be nonempty. We want to know for which $m \in \mathbb N$, given $n \in \mathcal A_k$, $mn$ is itself in $\mathcal A_k$: this will give the conclusion, once we know that all the numbers in the sequence are multiples of a smallest number $c(k)$ which belongs itself to $\mathcal A_k$. The proof of this latter fact is deferred to Theorem \ref{char2} since it fits better within that setting.

Suppose we have $n \in \mathcal A_k$, and take $m=p_1^{a_1} \cdots p_w^{a_w}$ with $a_i>0$ for each $i$; set $n=p_1^{\lambda_1}\cdots p_w^{\lambda_w} n'$ with $n'$ coprime to $m$ and $\lambda_i \geq 0$ for each $i$. Then one has
\[k=\frac{n}{z(n)}=\frac{p_1^{\lambda_1}\cdots p_w^{\lambda_w} n'}{z(p_1^{\lambda_1}\cdots p_w^{\lambda_w} n')}\]\[=\frac{p_1^{\lambda_1}\cdots p_w^{\lambda_w} n'}{\lcm(p_1^{\max(\lambda_1-e(p_1),0)}z(p_1), \ldots, p_w^{\max(\lambda_w-e(p_w),0)}z(p_w), z(n'))}.  \]
The $p_i$-adic valuation of this expression is $v_{p_i}(k)$, so in the denominator either $\max(\lambda_i-e(p_i),0)$ is the greatest power of $p_i$, or some of $z(p_1), \ldots, z(p_w), z(n')$ has $p$-adic valuation $\lambda_i-v_{p_i}(k) \geq \lambda_i-e(p_i)$. Furthermore, one has $\lambda_i \geq v_{p_i}(k)$, as $n$ has to be a multiple of $k$. 

Now, the number
\[ \frac{mn}{z(mn)}=\frac{p_1^{\lambda_1+a_1} \cdots p_w^{\lambda_w+a_w} n'}{z(p_1^{\lambda_1+a_1} \cdots p_w^{\lambda_w+a_w} n')}\]\[=\frac{p_1^{\lambda_1+a_1} \cdots p_w^{\lambda_w+a_w} n'}{\lcm(p_1^{\lambda_1+a_1-e(p_1)}z(p_1),\ldots,p_w^{\lambda_w+a_w-e(p_w)}z(p_w),z(n'))} \]
is equal to $k$ if and only if its $p_i$-adic valuation is $v_{p_i}(k)$ for each $i$, that is
\[v_{p_i}(k)=\lambda_i+a_i-\max(\lambda_i+a_i-e(p_i),\lambda_i-v_{p_i}(k)).  \]

Suppose that the first term in the $\max$ is the greater, that is $a_i\geq e(p_i)-v_{p_i}(k)$. The above equality reduces to $v_{p_i}(k)=e(p_i)$: so in this case each value $\geq e(p_i)-v_{p_i}(k)$ for $a_i$, and each nonnegative value for $\lambda_i$ is admissible, if $v_{p_i}(k)=e(p_i)$, and no value is admissible if $0 \leq v_{p_i}(k)<e(p_i)$.

Suppose that the second term is the greater, that is $a_i< e(p_i)-v_{p_i}(k)$. The equality reduces to $a_i=0$, which is impossible.

Starting from $c(k)$ and building all the members of $\mathcal A_k$ by progressively adding prime factors, we find exactly the statement of the theorem.
\end{proof}

In the remainder of this section, we show that $c(k)$ admits a more explicit description.
\begin{thm} \label{char2}
$c(k)=k \lcm \left \{ z^i(k) \right \}_{i=1} ^\infty$.
\end{thm}
\begin{proof}
To prove first that such an expression is well-defined,  we show that the sequence of iterates of $z$ eventually hits a fixed point.

First note that, for $k=\prod_i p_i^{\alpha_i}$, $z(k)=\lcm \left \{  p_i^{\max(\alpha_i-e(p_i),0)} z(p_i)\right \}_i$: this is a divisor of $\displaystyle \frac k {\rad (k)} \lcm \left \{ z(p_i)\right \}_i$, where $\rad(k)=\prod_i p_i$ is the radical of $k$. Consider now the largest prime factor $P$ of $k$: if $P \geq 7$, its exponent in the previous expression decreases by at least $1$ at each step, since the largest prime factor of $z(P)$ is strictly smaller than $P$. Consequently, after at most $v_P(k)$ steps, the exponent of $P$ would have vanished. By iterating the argument concerning the largest prime factor at each step, after a finite number $\ell$ of steps, $z^\ell (k)$ will have only prime factors smaller than $7$; set $z^\ell (k)=2^a 3^b 5^c$.

Recall now Theorem 1.1 of \cite{M1}: the fixed points of $z$ are exactly the numbers of the form $5^f$ and $12 \cdot 5^f$. By noting that $z(2^a)= 3 \cdot 2^{a-2}$, $z(3^b)=4 \cdot 3^{b-1}$, $z(5^c)=5^c$, we get that $z(2^a3^b5^c)=2^{\max(a-2,2)}3^{\max(b-1,1)}5^{c}$. Since we can continue this until $a \leq 2$ and $b \leq 1$, we are left with a few cases to check to show that the sequence of iterates indeed reaches a fixed point.

As $c(k)$ must be a multiple of $k$, call $T:=c(k)/k$. Consider next the obvious equalities
 \begin{align*} 
T & =   z(kT), \\ 
  z(T) & =    z^2(kT), \\ 
 z^2(T) & =  z^3(kT),\\ 
 & \hspace{2.4mm}  \vdots 
\end{align*} 
Write $x \overset{\text{div}}{\leftarrow} y$ for the statement ``$y$ divides $x$''. Then we have that
\begin{align*}
T &= z(kT)\\
 &\overset{\text{div}}{\leftarrow} z(\lcm(k,T))\\
 &= \lcm(z(k),z(T))\\
 &= \lcm(z(k),z^2(kT))\\
 &\overset{\text{div}}{\leftarrow} \lcm(z(k),z(\lcm(z(k),z(T))))\\ 
 &= \lcm(z(k),\lcm(z^2(k),z^2(T)))\\
 &= \lcm(z(k), z^2(k), z^3(kT))\\
 & \vdots \\
 &= \lcm(z(k),z^2(k),z^3(k), \dots).
 \end{align*} 
Note that we have not used yet that $kT$ is the smallest member of $\mathcal A_k$; this means the above reasoning works for any member of $\mathcal A_k$, so that any number in $\mathcal A_k$ is a multiple of $k\lcm(z(k),z^2(k),z^3(k), \dots)$. If we manage to prove that $k\lcm(z(k),z^2(k),z^3(k), \dots)$ is indeed in the sequence, we will obtain the divisibility argument we needed in the proof of Theorem \ref{char1}.

Thus, we want to prove that $T=\lcm \left \{ z^i(k) \right \}_{i=1} ^\infty$ works; it is enough to prove that the divisibilities we previously derived are equalities, or in other words that $z(kT)=z(\lcm(k,T))$ for $T$ defined this way.

If $k=\prod _i p_i^{\alpha_i}$ with $\alpha_i \leq e(p_i)$ for each $i$, then
\[
\begin{split}
T &= \lcm \left (z \left (\prod_i p_i ^{\alpha_i}\right ),\  z^2\left (\prod_i p_i^{\alpha_i}\right ),\  \ldots\right )\\
 &= \lcm\left (\lcm \left \{ z\left (p_i^{\alpha_i}\right ) \right \}_i , \ \lcm \left \{ z^2\left (p_i^{\alpha_i}\right ) \right \}_i,\  \ldots\right)\\
 &= \lcm\left ( \left \{ z\left (p_i^{\alpha_i}\right) \right \}_i ,\   \left \{ z^2\left(p_i^{\alpha_i}\right) \right \}_i,\  \ldots\right)\\
\end{split}
\]
and
\[ z(kT) = z\left(\left ( \prod_i p_i^{\alpha_i}  \right ) \lcm\left(z\left(\prod_i p_i^{\alpha_i}\right),\  z^2\left(\prod_i p_i^{\alpha_i}\right),\  \ldots\right)\right).\]
We would like to bring the $\prod_i p_i^{\alpha_i}$ into the least common multiple, but some power of $p_i$ could divide the iterated entry point of some other prime to a higher power. Define then $m(p_h)$ to be the largest exponent of a power of $p_h$ that divides $z^i(p_j)$ as $i$ and $j$ vary; thus
\[
\begin{split}
&z\left(\left (\prod_i p_i^{\alpha_i} \right ) \lcm\left(z\left (\prod_i p_i^{\alpha_i}\right),\  z^2\left (\prod_i p_i^{\alpha_i}\right),\  \ldots\right)\right ) \\=& z\left(\lcm\left(\left \{ p_i^{m\left(p_i\right)+\alpha_i} \right \}_i,\  z\left (\prod_i p_i^{\alpha_i}\right ),\  z^2\left(\prod_i p_i^{\alpha_i}\right ),\  \ldots\right )\right )
 \\=& \lcm\left (\left \{ z\left(p_i^{m\left (p_i\right)+\alpha_i}\right ) \right \}_i,\  \left \{ z^2\left(p_i^{\alpha_i}\right) \right \}_i,\   \left \{ z^3\left(p_i^{\alpha_i}\right ) \right \}_i,\  \ldots\right ).   
\end{split}
\]
We need this to be equal to
\[ \lcm\left (\left \{ z\left(p_i^{\alpha_i}\right) \right \}_i,\  \left \{ z^2\left(p_i^{\alpha_i}\right) \right \}_i,\   \left \{ z^3\left(p_i^{\alpha_i}\right) \right \}_i,\  \ldots\right )=\lcm(z(k),z(T))=z(\lcm(k,T)).\]
All that is left to do now is to remark that this is true if and only if their $p_i$-adic valuations are equal for each $i$, or in other words, as $p_i$ is coprime to $z(p_i^{\alpha_i})=z(p_i)$,
\[\max(m(p_i)+\alpha_i-e(p_i)),m(p_i))=m(p_i),  \]
and this is evident.
\end{proof}

\section{The proof of Theorem \ref{main}}

Let $x\ge 10$. One of our ingredients is the following result from \cite{GP}.

\begin{lemma}[\cite{GP}, Theorem 3]
\label{lemma:GP}  As $x \rightarrow \infty$,
\[ \# \left \{ n \leq x : z(n)=m \right \} \leq x^{1-(1/2+o(1))\log\log\log x/\log\log x}, \]
uniformly in $m$.
\end{lemma}

Let $n\in {\mathcal A}(x)$. By Theorem \ref{char1}, every self-Fibonacci divisor is of the form $c(k) m$, where $m$ is composed of primes that divide $k$. Thus, write $n=c(k) m$, where every prime factor of $m$ divides $k$. Let $C(x):=x^{\log\log\log x/\log \log x}$. We distinguish two cases.

\medskip

{\bf Case 1.} $k\le x/C(x)$.

\medskip

Let ${\mathcal A}_1(x)$ be the subset of such $n\in {\mathcal A}(x)$. We fix $k$ and count possible $m$'s because $c(k)$ is determined by $k$; we use an idea similar to the one of the proof of Theorem 4 in \cite{ELP}. Clearly, $m$ has at most $\omega(k)$ distinct prime factors. Define next $\Psi(x,y)$ to be the number of positive integers $\ell\le x$ whose largest prime factor $P(\ell)$ satisfies the inequality $P(\ell) \le y$, and let $p_s$ be the $s$-th prime. If $\mathcal P_k$ is the set of the prime divisors of $k$, the quantity of numbers $m\le x$ all of whose prime factors are in $\mathcal P_k$ is of course at most $ \Psi(x,p_{\omega(k)})\leq \Psi(x,2\log x)$ for $x$ large enough. Here we used the fact that $p_s<s(\log s+\log\log s)$ for all $s\ge 6$ (Theorem 3 of \cite{RS}) together with $\omega(k)<2\log k/\log\log k$ for all $k\geq 3$. Classical estimates on $\Psi(x,y)$, such as the one of de Bruijn (see, for example, Theorem 2 on page 359 in \cite{Ten}), show that if we put
$$
Z:=\frac{\log x}{\log y}\log\left(1+\frac{y}{\log x}\right)+\frac{y}{\log y} \log\left(1+\frac{\log x}{y}\right),
$$
then the estimate 
\begin{equation}
\label{eq:dB}
\log\Psi(x,y)= Z\left(1+O\left(\frac{1}{\log y}+\frac{1}{\log\log (2x)}\right)\right)
\end{equation}
holds uniformly in $x\ge y\ge 2$.  The above estimates \eqref{eq:dB} with $y=2\log x$ imply that there are at most $C(x)^{(3 \log 3-2 \log 2+o(1))/ \log \log \log x}=C(x)^{o(1)}$ values of $m$ for any fixed $k$. Summing up over $k$, we get that
\begin{equation}
\label{eq:A1}
\# {\mathcal A}_1(x)\leq C(x)^{o(1)}\sum_{k\leq x/C(x)} 1\le  \frac{x}{C(x)^{1+o(1)}}\qquad (x\text{ large}).
\end{equation}

\medskip

{\bf Case 2.} $x/C(x)<k\le x.$

\medskip

Here, we have $kz(k)\le c(k)\le x$, whence $z(k)\le C(x)$. Fix $z(k)=z$ in $[1,C(x)]$. By Lemma \ref{lemma:GP}, if we put ${\mathcal B}_z:=\{n \in \mathbb N: z(n)=z\}$, then the inequality
\begin{equation}
\label{eq:B}
\# \mathcal B_z(t) \le t/C(t)^{1/2+o(1)}\qquad {\text{\rm holds~as}}\quad t\to\infty.
\end{equation}
We now let $k\in {\mathcal B}_z$. Then $n\le x$ is a multiple of $k z$. The  number of such $n$ is $\lfloor x/kz\rfloor\le x/kz$. Summing up the above inequality over $k\in {\mathcal B}_z$ and using partial summation and \eqref{eq:B}, we have
\begin{eqnarray*}
\frac{x}{z}\sum_{\substack{k\in {\mathcal B}_z\\ x/C(x)<k\le x}} \frac{1}{k} & = & \frac{x}{z}\int_{x/C(x)}^x \frac{\text{d} \# {\mathcal B}_z(t)}{t} \\ 
& = & \frac{x}{z}\left(\frac{\# {\mathcal B}_z(t)}{t}\Big|_{t=x/C(x)}^{t=x} +\int_{x/C(x)}^x \frac{\# {\mathcal B}_z(t)}{t^2} \text{d} t\right)\\
& \leq & \frac{x}{z} \left(\frac{\# {\mathcal B}_z(x)}{x}+\int_{x/C(x)} ^x \frac{\text{d} t}{t C(t)^{1/2+o(1)}}\right)\\
& = & \frac{x}{z}\left(\frac{1}{C(x)^{1/2+o(1)}}+\frac{1}{C(x)^{1/2+o(1)}} \int_{x/C(x)}^x \frac{\textnormal{d}t}{t}\right)\\
& = & \frac{(1+o(1))x \log C(x)}{z C(x)^{1/2+o(1)}}=  \frac{x}{z C(x)^{1/2+o(1)}},
\end{eqnarray*}
where in the above calculation we used the fact that 
$$
C(t)^{1/2+o(1)}=C(x)^{1/2+o(1)}\quad {\text{\rm uniformly~in}}\quad t\in [x/C(x),x]\quad {\text{\rm as}}\quad x\to\infty.
$$
We now sum over $z\in [1,C(x)]$, and obtain that
\begin{multline}
\label{eq:A2}
\# {\mathcal A}_2(x)\le \frac{x}{C(x)^{1/2+o(1)}}\sum_{1\le z\le C(x)} \frac{1}{z}\\ = \frac{(1+o(1))x\log C(x)}{C(x)^{1/2+o(1)}}=\frac{x}{C(x)^{1/2+o(1)}}\qquad (x\to\infty).
\end{multline}
The desired conclusion now follows from \eqref{eq:A1} and \eqref{eq:A2}.

\section{Comments} 

Of course, the methods we presented apply equally well to other Lucas sequences, where analogues of Theorems \ref{char1}, \ref{char2} and \ref{main} hold; we chose to display the Fibonacci case, when the classification takes a particularly simple form.

To conclude, we make some observations to promote future progress. The problem of finding lower bounds for $A(x)$ requires completely different ideas; one can prove that \[\log A(x)=\log \# \left \{ n \leq x : c(n) \leq x,\text{ $n$ squarefree} \right \}+O\left ( \frac{\log x \log\log\log x}{\log\log x}  \right ),  \]
so that in order to prove $A(x)=x^{1+O(\log\log\log x/\log\log x)}$ unconditionally one would need to build many squarefree $n$ with small $c(n)$. The best we managed to prove is  that $\log c(n) <3 P(n)$ (by double counting), and $\log c(n) <7 \sum_{p|n}(\log p)^2$ (by induction), but neither of these is sufficient. This hints at building numbers $n$ for which their prime factors share most of their Pratt-Fibonacci trees (Pratt trees built with the factors of $z(p)$ as children of a node $p^\delta$, taken with their exponents).

The set of numbers $n$ with small $c(n)$ is both small and large in a certain sense: it has asymptotic density $0$ and exponential density $1$, conjecturally.

It is indeed likely that $c(n)$ is quite large for most $n$. Recall that putting 
\[
F(n):={\text{\rm rad}} \left(\prod_{k\ge 1} \phi^{k}(n)\right),
\]
then in \cite{LuPo} it is proved that the inequality
\[
F(n)>n^{(1+o(1))\log\log n/\log\log\log n}
\]
holds for $n$ tending to infinity through a set of asymptotic density $1$. Since $c(n)$ is quite similar to $F(n)$, we conjecture that a similar result holds for $c(n)$ as well.

\section*{Acknowledgements}
We thank K. Ford, P. Leonetti, D. Marques, C. Pomerance, G. Tenenbaum for their help.

\end{document}